\documentclass[11pt]{article}
\usepackage{iftex}
\ifPDFTeX
\pdfoutput=1
\fi

\usepackage{amsthm}
\PassOptionsToPackage{pdfauthor={Rohit Agrawal}}{hyperref}
\usepackage[stix]{rohitagr}
\usepackage[margin=1.2in]{geometry}

\usepackage[capitalize]{cleveref}

\newtheorem{theorem}{Theorem}[section]
\newtheorem{proposition}[theorem]{Proposition}
\newtheorem{lemma}[theorem]{Lemma}
\newtheorem{corollary}[theorem]{Corollary}
\newtheorem{conjecture}[theorem]{Conjecture}

\theoremstyle{definition}
\newtheorem{definition}[theorem]{Definition}

\theoremstyle{remark}
\newtheorem{remark}[theorem]{Remark}

\DeclareMathOperator\KL{D}
\DeclareMathOperator\Cov{Cov}
\DeclareMathOperator\Var{Var}
\DeclareMathOperator\Multi{Multi}

\usepackage[backend=biber,style=alphabetic,maxalphanames=6,maxbibnames=99,giveninits=true,backref=true]{biblatex}
\newcommand\ind[1]{\mathbf{1}_{#1}}

\newcommand\N{\mathbf{N}}
\newcommand\R{\mathbf{R}}

\newcommand\intd[1]{\, \mathrm{d}#1}

\newcommand\empkl[3]{V_{#1,#2,#3}}
\newcommand{\centered}[2][\E]{#2-#1{#2}}

\usepackage{tikz}
\usepackage{accsupp}
\usetikzlibrary{svg.path}
\definecolor{orcidlogocol}{HTML}{A6CE39}
\newcommand{\orcidlogo}{\BeginAccSupp{method=escape,Alt={ORCiD ID}}\begin{tikzpicture}[yscale=-1,transform shape]
\fill[orcidlogocol] svg{M256,128c0,70.7-57.3,128-128,128C57.3,256,0,198.7,0,128C0,57.3,57.3,0,128,0C198.7,0,256,57.3,256,128z};
    \fill[white] svg{M86.3,186.2H70.9V79.1h15.4v48.4V186.2z}
                 svg{M108.9,79.1h41.6c39.6,0,57,28.3,57,53.6c0,27.5-21.5,53.6-56.8,53.6h-41.8V79.1z M124.3,172.4h24.5c34.9,0,42.9-26.5,42.9-39.7c0-21.5-13.7-39.7-43.7-39.7h-23.7V172.4z}
                 svg{M88.7,56.8c0,5.5-4.5,10.1-10.1,10.1c-5.6,0-10.1-4.6-10.1-10.1c0-5.6,4.5-10.1,10.1-10.1C84.2,46.7,88.7,51.3,88.7,56.8z};
\end{tikzpicture}\EndAccSupp{}}
\newcommand\orcidlogosized[1]{\raisebox{-#1/5}{\resizebox{!}{#1}{\orcidlogo}}}
\newcommand\orcid[2][1em]{\mbox{\href{https://orcid.org/#2}{\orcidlogosized{#1}\hspace{\dimexpr #1/2}\nolinkurl{#2}}}}

\title{Finite-sample concentration of the empirical relative entropy
around its mean}
\author{Rohit Agrawal\\\orcid{0000-0001-5563-7402}}

\begin{document}
\maketitle
\begin{abstract}
  In this note, we show that the relative entropy of an empirical distribution
of $n$ samples drawn from a set of size $k$ with respect to the true underlying
distribution is exponentially concentrated around its expectation, with central
moment generating function bounded by that of a gamma distribution with shape
$2k$ and rate $n/2$. This improves on recent work of
\textcite{bha_pen_sharp_2021} on the same problem, who showed such a similar
bound with an additional polylogarithmic factor of $k$ in the shape, and also
confirms a recent conjecture of \textcite{mar_jia_tan_now_wei_concentration_2020}.
The proof proceeds by reducing the case $k>3$ of the multinomial distribution to
the simpler case $k=2$ of the binomial, for which the desired bound follows from
standard results on the concentration of the binomial.
\end{abstract}

\section{Introduction}

Given $n$ samples from some distribution $P = (p_{1}, \dots, p_{k})$ on a finite
set of size $k$, the realized fraction of samples corresponding to each element
$\of{\frac{X_{1}}{n}, \dots, \frac{X_{k}}{n}}$ is a natural estimator (and
in fact the maximum likelihood estimator) of the underlying distribution $P$.
Since the Neyman--Pearson lemma \cite{ney_pea_problem_1933}
reduces optimal hypothesis testing to understanding the distribution of the
likelihood ratio statistic, in this case we are led to study the
empirical relative entropy with respect to the true
distribution:
\begin{definition}
  Given a distribution $P=(p_{1}, \dots, p_{k})$ on a finite set of size $k$ and
  multinomially distributed random variables $(X_{1}, \dots, X_{k}) \sim \Multi\of{n; (p_{1}, \dots, p_{k})}$
  for a positive integer $n$, the \emph{empirical relative entropy} is
  \[
    \empkl nkP
    = \KL\diver{\of{\frac{X_{1}}{n},\dots,\frac{X_{k}}{n}}}{\of{p_{1},\dots,p_{k}}}
    =
    \sum_{i=1}^{k} \frac{X_{i}}{n}\log\frac{X_{i}}{np_{i}}
    \,,
  \]
  and is such that $2n\empkl nkP$ is the likelihood-ratio statistic of the hypothesis
  that the probabilities of $X$ are $P$, where
  \[
    \KL\diver[\big]{\of{q_1,\dotsc,q_k}}{\of{p_1,\dotsc,p_k}}
    = \sum_{i=1}^{k} q_{i}\log\frac{q_{i}}{p_{i}}
  \]
  denotes the \emph{relative entropy} or \emph{Kullback--Leibler (KL) divergence}
  of $Q$ with respect to $P$.\footnote{All logarithms and exponentials
  are in the natural base.}
\end{definition}

For fixed $k$, it is well known (from e.g.~Wilks' theorem \cite{wilks_largesample_1938})
that the likelihood ratio statistic $2n\empkl nkP$ converges in distribution to
$\chi^{2}_{k-1}$ a chi-squared distribution with $k-1$ degrees of freedom as
$n$ goes to infinity (assuming $P$ is not supported on a set of size smaller
than $k$). For the specific case we are interested in of the multinomial
distribution, there are also known finite-sample bounds, most notably the
now-standard bound obtained via the method of types \cite{csiszar_method_1998}
that
\[\PR{\empkl nkP\geq \eps}\leq \binom{n+k-1}{k-1}\cdot\exp\of{-n\eps}\]
for all real $\eps\geq0$, which has optimal decay in $\eps$ as $n$ and $\eps$ go to infinity,
but is trivial for $\eps$ close to $\E{\empkl nkP}\leq \log\of{1 + \frac{k-1}{n}}\leq\frac{k-1}{n}$
\cite{paninski_estimation_2003}. \textcite{mar_jia_tan_now_wei_concentration_2020}
recently substantially improved this bound, giving a roughly quadratic
improvement in the $\binom{n+k-1}{k-1}$ factor while maintaining the decay in $\eps$,
and also posed several conjectures about improved bounds. Subsequently,
the author gave an incomparable exponential bound \cite{agrawal_finitesample_2020} (further
improved by \textcite{guo_ric_chernofftype_2021}) which becomes non-trivial for
$\eps > \frac{k-1}{n}$ but which has decay like $n\eps\cdot(1-o(1))$ for large $\eps$,
by bounding the moment generating function of $\empkl nkP$.

Most of the above bounds focused on the question of bounding the probability that
$\empkl nkP$ exceeds $0$ by some $\eps$, but it is also natural to ask about
concentration around $\E{\empkl nkP}$. In particular, \textcite{mar_jia_tan_now_wei_concentration_2020}
posed the following conjecture:
\begin{conjecture}[{\cite[Conjecture 2]{mar_jia_tan_now_wei_concentration_2020}}]
  \label{conj:tb}
  There are positive constants $c_1$ and $c_2$ such that for every $n$, $k$,
  $P$, and $\eps\geq 0$ it holds that
  \[
    \PR{\abs{\centered{\empkl nkP}} \geq \eps}
    \leq
    c_1\exp\of{-c_2 \min\set{\frac{n^2\eps^2}{k-1},n\eps}}
    \,.
  \]
\end{conjecture}
By standard results on subgamma random variables (e.g.~\cite[\S2.4, Theorem
2.3]{bou_lug_mas_concentration_2013}), \cref{conj:tb} is equivalent to upper
bounding the central moment generating function of $\empkl nkP$ by that of a
gamma distribution with shape $C_{1}(k-1)$ and rate $C_{2}n$ on a ball around the
origin of radius $C_{3}n$ for positive constants $C_{1}$, $C_{2}$, and $C_{3}$.

Similarly, in \cite{agrawal_deriving_2020}, the author conjectured that the bound on the
non-centered moment generating function of $\empkl nkP$ by that of the gamma
distribution with shape $k - 1$ and rate $n$ on the positive reals
\cite[Theorem I.3]{agrawal_finitesample_2020} also holds for the centered
version (i.e.~with constants $C_1=C_2=1$, but for the positive reals), which from the
above would suffice to prove a one-sided version of \cref{conj:tb}:
\begin{conjecture}[{\cite[Conjecture 4.4.5]{agrawal_deriving_2020}}]
  \label{conj:mgf}
  For every $n$, $k$, and $P$, we have for all $0\leq t<n$ that
  \[
    \log\E{\exp\of{t\of{\centered{\empkl nkP}}}}
    \leq
     (k-1) \log\of{\frac{\exp\of{-t/n}}{1-t/n}}\,,
  \]
  where the right-hand side is the centered log moment generating function
  of a gamma distribution of shape $k-1$ and rate $n$.
\end{conjecture}

Significant progress towards these conjectures was made in recent work of
\textcite{bha_pen_sharp_2021}, who established near-optimal bounds in the
case that the probabilities of the multinomial distribution are bounded
away from $0$:
\begin{theorem}[Equivalent form of {\cite[Theorem 1]{bha_pen_sharp_2021}}]
  \label{thm:bha_pen_sharp_2021}
  There are positive constants $C_{1}$ and $C_{2}$ such that for
  all $n$, $k$, and $P$, it holds that
  \[
    \log\E{\exp\of{t\of{\centered{\empkl nkP}}}}
    \leq
     C_{1}\cdot k\log^{4}\of{\frac{k}{\min_{i} p_{i}}}\cdot \log\of{\frac{\exp\of{-C_{2}t/n}}{1-C_{2}t/n}}\,,
  \]
  for all $\abs t < C_{2}n$.
\end{theorem}
However, \cref{thm:bha_pen_sharp_2021} does not suffice to prove \cref{conj:tb}
due to the additional polylogarithmic factors in $k$ and $\min_{i} p_{i}$.

In this work, we close this gap, proving \cref{conj:tb} by giving an upper
bound on the centered moment generating function of $\empkl nkP$ with
explicit constants, though falling short of those conjectured by
\cref{conj:mgf}.
\begin{theorem}[This work, main result]
  \label{thm:mgf}
  For every $n$, $k$, and $P$, we have for all $t<n/2$ that
  \[
    \log\E{\exp\of{t\of{\centered{\empkl nkP}}}}
    \leq
    \min\set{
      \frac{4kt^2/n^2}{1-2t/n} 
      ,\,
      2k \log\of{\frac{\exp\of{-2t/n}}{1-2t/n}}
    }
    \,,
  \]
  and so in particular the centered cumulant generating function of
  $\empkl nkP$ is bounded by that of
  a gamma distribution with shape $2k$\footnote{In fact, the techniques
    in this work are capable of establishing shape $Ck$ for a constant
    $1<C<2$, see \cref{rmrk:expo-const}. More generally, here and in the corollaries
    we give explicit values of constants, but have not tried to optimize them.}
  and rate $n/2$.
\end{theorem}
\begin{remark}\label{rmrk:negative-mgf}
  At first glance the validity of the above bound for the entire negative real
  line rather than an interval of length $Cn$ appears qualitatively stronger
  than what is necessary for \cref{conj:tb}, but in fact such bounds are
  equivalent because $\empkl nkP\geq 0$ implies the trivial upper bound
  for $t\leq 0$ that
  \[
  \log\E{\exp\of{t\of{\centered{\empkl nkP}}}} \leq \abs{t}\cdot \E{\empkl nkP}
  \leq \abs{t}\cdot \frac{k-1}{n}
  \,,
  \]
  which already establishes the claim for $t\leq -Cn$ for any positive constant
  $C$, and in fact is stronger than \cref{thm:mgf} for sufficiently negative $t$.
\end{remark}
\begin{corollary}\label{cor:tail}
  For all $\eps\geq 0$ we have that
  \begin{multline*}
    \PR{\empkl nkP \geq \E{\empkl nkP} + \eps}
    \leq \of{1 + \frac{n\eps}{4k}}^{2k}\cdot \exp\of{-\frac{n\eps}2}
    \\
    \leq \exp\of{-\frac{3n^2\eps^2}{48k+8n\eps}}
    \leq \exp\of{-\min\set{\frac{n^2\eps^2}{24k},\frac{n\eps}8}}
  \end{multline*}
  and for all $0\leq\eps \leq 2k/n$ we have that
  \[
    \PR{\empkl nkP \leq \E{\empkl nkP}-\eps}
    \leq \exp\of{-k\of{1-\sqrt{1-\frac{n\eps}{2k}}}^2}
    \leq \exp\of{-\frac{n^2\eps^2}{16k}}
    \,.
  \]
  In particular, \cref{conj:tb} holds with $c_1=2$ and $c_{2}=1/48$.
\end{corollary}
\begin{remark}
  For $\eps>2k/n$, \cref{thm:mgf} implies that $\PR{\empkl nkP \leq \E{\empkl
nkP} - \eps} = 0$, but since $\E{\empkl nkP}\leq \log\of{1+\frac{k-1}n}\leq
\frac{k-1}n$ \cite{paninski_estimation_2003}, as in \cref{rmrk:negative-mgf}
this (and the corresponding part of \cref{cor:tail}) is subsumed by the
fact that $\empkl nkP\geq 0$ implies $\PR{\empkl nkP \leq \E{\empkl nkP} - \eps}
= 0$ for all $\eps> \E{\empkl nkP}$.
\end{remark}

Similarly, \cref{thm:mgf} also implies moment bounds, recovering a weaker
version of a variance upper bound of
\cite{mar_jia_tan_now_wei_concentration_2020} and strengthening the moment
bounds from \cite{agrawal_finitesample_2020,bha_pen_sharp_2021}.
\begin{corollary}\label{cor:moments}
  We have that $\Var\of{\empkl nkP}\leq 8k/n^2$, and more generally, for
  all integers $q\geq 1$ we have
  \[
    \E{\of{\centered{\empkl nkP}}^{2q}}\leq
    \frac{2^{6q}\of{k^qq! + (2q)!}}{n^{2q}}
    \,,
  \]
  so that in particular for all real $q\geq 1$ we have that
  \[
    \sqrt[q]{\E{\abs{\centered{\empkl nkP}}^{q}}}\leq
    \frac{24}{n}\of{\sqrt{kq} + q}\,.
  \]
\end{corollary}

The proof of \cref{thm:mgf} follows a similar outline as that of the earlier
work \cite{agrawal_finitesample_2020} for the non-centered moment generating
function, later extended to the centered version by \cite{bha_pen_sharp_2021},
namely it reduces the multinomial case to the simpler case $k=2$ of the binomial
and then bounding the binomial. Our point of departure is in the reduction used:
the aforementioned works used a reduction that takes advantage of the dependence
between the variables $X_i$ for $i\in\set{1,\dots,k}$ and as a result has bounds
in terms of $k-1$, but does not adapt as easily to the centered case (though it
can be done, as in \cite{bha_pen_sharp_2021}); by contrast we use a reduction
that shows we can consider independent $X_i$ by incurring a quadratic loss,
resulting in a simpler proof and stronger bound in the centered case (though
weaker in the non-centered case, both via the quadratic loss and by depending on
$k$ rather than $k-1$). It would be interesting to find an approach that worked
cleanly for both cases without incurring these losses.

\section{Proof}
\subsection{Reduction from the multinomial to binomial}

In this section, we reduce the case of an alphabet of size $k$ to that of an
alphabet of size $2$. To state the result, it is convenient to work with a
slightly modified formulation of the relative entropy, commonly used when
considering it as an $f$-divergence \cite{csiszar_informationstheoretische_1963,morimoto_markov_1963,ali_sil_general_1966}.

\begin{definition}
  Let $\phi:\R_{\geq 0} \to \R_{\geq 0}$ be the function on the non-negative
  reals given by $\phi(x) = x\log x - x + 1$ for $x > 0$ and
  $\phi(0) = 1$, and let $\phi_+(x) = \phi(x)\cdot \ind{x\geq 1}$ and
  $\phi_-(x) = \phi(x)\cdot \ind{x\leq 1}$.
\end{definition}
\begin{lemma}
  \label{lem:phiprops}
  $\phi$ is continuous, convex, non-negative, decreasing on $[0,1]$, increasing
  on $[1,\infty)$, and has $\phi(1)=0$. In particular, $\phi_+$ and $\phi_-$ are
  also continuous, convex, non-negative, are respectively non-decreasing and
  non-increasing, and satisfy $\phi = \phi_+ + \phi_-$.
\end{lemma}
\begin{lemma}
  \label{lem:klphi}
  The relative entropy satisfies
  \[
    \KL\diver[\big]{\of{q_1,\dotsc,q_k}}{\of{p_1,\dotsc,p_k}}
    = \sum_{i=1}^k p_i\cdot \phi\of{\frac{q_i}{p_i}}
  \]
  where $0\cdot\phi(q_i/0)=\infty$ if $q_i > 0$.
\end{lemma}

With this definition, we can state the main result of this section:
\begin{proposition}\label{prop:multired}
  For all $n$, $k$, $P=(p_1,\dotsc, p_k)$, and $t\in \R$, it holds that
  \begin{multline*}
    \E{\exp\of{t\cdot \of{\centered{\empkl nkP}}}}
    \\
    \leq
    \prod_{i=1}^k\sqrt{
      \E{\exp\of{2t\cdot\of{\centered{p_i\phi_+\of{X_i/np_i}}}}}
    }
    \\
    \cdot
    \prod_{i=1}^k\sqrt{
      \E{\exp\of{2t\cdot\of{\centered{p_i\phi_-\of{X_i/np_i}}}}}
    }
  \end{multline*}
  where $(X_1,\dots,X_k)$ is multinomially distributed with $n$ samples
  and probabilities $P$.
\end{proposition} 

Note that the expectations in \cref{prop:multired} involve only a single $X_i$
at a time, i.e.~we have broken their dependence. To do so, we use the fact
that the variables are \emph{negatively associated} in the sense of
\textcite{joa_pro_negative_1983}.

\begin{definition}[{\cite[Definition 2.1]{joa_pro_negative_1983}}]
  \label{def:na}
  A collection of real-valued random variables $(Z_1, \dots, Z_m)$ is said to be
  \emph{negatively associated} if for all disjoint subsets $A_1,A_2
  \subseteq\set{1,\dots,m}$ and (pointwise) non-decreasing functions
  $f_i:\R^{\abs{A_i}}\to\R$, it holds that $\Cov\of{f_1(X_i, i\in A_1),
    f_2(X_j, j\in A_2)}\leq 0$.
\end{definition}
\begin{lemma}[{\cite[Properties $P_2$ and $P_6$]{joa_pro_negative_1983}}]
  \label{lem:naprops}
  If $(Z_1, \dots, Z_m)$ are negatively associated random variables, then
  for all functions $f_1,\dots,f_m:\R\to\R$ which are either all
  non-increasing or non-decreasing, the random variables
  $\of{f_1(Z_1), \dots, f_m(Z_m)}$ are negatively associated.
  In particular, if each $f_i(Z_i)\geq 0$ almost surely, then
  $\E{f_1(Z_1)\cdots f_m(Z_m)}\leq \E{f_1(Z_1)}\cdots \E{f_m(Z_m)}$.
\end{lemma}
\begin{lemma}[{\cite[3.1(a)]{joa_pro_negative_1983}}]
  \label{lem:namultinom}
  For all positive integers $n$, $k$ and probabilities $P=(p_1,\dots,p_k)$,
  the random variables $(X_1,\dots,X_k)$ distributed multinomially with
  $n$ samples and probabilities $P$ are negatively associated.
\end{lemma}

We would like to apply \cref{lem:naprops} to the KL divergence, but cannot do
so directly since the function $\phi$ is not monotone; however, since $\phi$
can be written as the sum of the monotone functions $\phi_+$ and $\phi_-$,
we can apply it after first separating the two parts by Cauchy--Schwarz,
incurring a quadratic penalty.

\begin{proof}[{Proof of \cref{prop:multired}}]
  Fix $n$, $k$, $P=(p_1,\dotsc,p_k)$, and $t\in\R$. Then
  by \cref{lem:klphi} and linearity of expectation we have that
  \[\centered{\empkl nkP} = \sum_{i=1}^k\centered{p_i\phi_+(X_i/np_i)}
    + \sum_{i=1}^k \centered{p_i\phi_-(X_i/np_i)}\,,\]
  and so by Cauchy--Schwarz we have
  \begin{multline*}
    \E{\exp\of{t\cdot \of{\centered{\empkl nkP}}}}
    \\
    \leq
    \sqrt{
      \E{\prod_{i=1}^k\exp\of{2t\cdot\of{\centered{p_i\phi_+\of{X_i/np_i}}}}}
    }\\\cdot\sqrt{
      \E{\prod_{i=1}^k\exp\of{2t\cdot\of{\centered{p_i\phi_-\of{X_i/np_i}}}}}
    }
    \,.
  \end{multline*}
  Now, since $\phi_+$ and $\phi_-$ are monotone (\cref{lem:phiprops}),
  we have that the functions
  \begin{align*}
    f_i(x)&=\exp\of{2t\of{p_i\phi_+(x/np_i) - \E{\phi_+(X_i/np_i)}}}\\
    g_i(x)&=\exp\of{2t\of{p_i\phi_-(x/np_i) - \E{\phi_-(X_i/np_i)}}}
  \end{align*}
  are for each $i$ respectively non-decreasing and non-increasing
  if $t\geq 0$ and respectively non-increasing and non-decreasing
  if $t\leq 0$. Thus, since the exponential function is non-negative,
  the result follows \cref{lem:naprops,lem:namultinom}.
\end{proof} 

\subsection{Bounding the binomial}

It remains to bound the centered moment generating function of the random
variables $np\cdot \phi_*(X/np)$ for $X$ binomially distributed with $n$ trials
and success probability $p$, where $\phi_*\in\set{\phi_+,\phi_-}$. Such bounds
can be derived using standard results on subgamma random variables (as done in
\cite{bha_pen_sharp_2021} following \cite[\S2.4]{bou_lug_mas_concentration_2013}), but
we do so explicitly here both for completeness and to derive (less-standard)
bounds in terms of the moment generating function of the gamma distribution
itself for comparison to \cref{conj:mgf}.

To begin, we use the standard fact that these
(non-centered) random variables satisfy strong tail bounds, via e.g.\ the
classical Hoeffding inequality:

\begin{lemma}
  \label{lem:halfkl-domexpo}
  If $X$ is binomially distributed with $n$ trials of success probability $p$,
  then the random variables
  \begin{align*}
    Z_+ &= n\KL\diver{\of{\frac Xn, 1-\frac Xn}}{\of{p,1-p}}\cdot \ind{X\geq np}\\
    Z_- &= n\KL\diver{\of{\frac Xn, 1-\frac Xn}}{\of{p,1-p}}\cdot \ind{X\leq np}
  \end{align*}
  are both stochastically dominated by the exponential distribution, that is,
  $\PR{Z_i \geq x}\leq \exp(-x)$ for all $x\geq 0$ and $i\in\set{+,-}$.
\end{lemma}
\begin{proof}
  By Hoeffding's inequality \cite{hoeffding_probability_1963} we have that for
  any real $k\geq np$ (resp.\ $k\leq np$) it holds that $\PR{X\geq k}$ (resp.\
  $\PR{X\leq k}$) is at most $\exp\of{-n\cdot f(k)}$ for $f(k) = \KL\diver{\of{\frac kn, 1-\frac
        kn}}{\of{p,1-p}}$,
    so since $f(k)$ is increasing in $k$ for $k\geq np$ and decreasing in $k$ for $k\leq np$, we get
    by inverting $f$ that
  $\PR{n\cdot f(X)\cdot \ind{X\geq np}
    \leq \eps}\leq \exp\of{-n\eps/n} = \exp\of{-\eps}$
  as desired, and analogously for the other tail.
\end{proof} 
\begin{corollary}
  \label{cor:redqty-domexpo}
  If $X$ is binomially distributed with $n$ trials of success probability $p$,
  then $np\cdot \phi_+(X/np)$ and $np\cdot\phi_-(X/np)$ are both stochastically
  dominated by an exponential random variable.
\end{corollary}
\begin{proof}
  We have that
  \begin{align*}
    n\KL\diver{\of{\frac Xn, 1-\frac Xn}}{\of{p,1-p}}\cdot \ind{X\geq np}
    &=
    np\cdot\phi_+\of{\frac X{np}}
    +
      n(1-p)\cdot\phi_-\of{\frac{n-X}{n(1-p)}}\\
    &\geq np\cdot\phi_+\of{\frac X{np}} \geq 0\\
    n\KL\diver{\of{\frac Xn, 1-\frac Xn}}{\of{p,1-p}}\cdot \ind{X\leq np}
    &=
      np\cdot\phi_-\of{\frac X{np}}
      +
      n(1-p)\cdot\phi_+\of{\frac{n-X}{n(1-p)}}\\
    &\geq np\cdot\phi_-\of{\frac X{np}} \geq 0
  \end{align*}
  so that the result follows from \cref{lem:halfkl-domexpo}.
\end{proof}

Finally, we show that random variables satisfying such tail bounds
have their centered moment generating function bounded by that of a
gamma distribution.
\begin{lemma}
  \label{lem:nonneg-cgf-rep}
  Let $Z$ be a non-negative random variable. Then for all $t\in\R$, we
  have that
  \[
    \log\E{\exp\of{t\of{\centered Z}}}
    = \log\of{1 + t\E Z + \int_0^\infty t\of{\exp\of{tx}-1}\PR{Z\geq x} \intd x}
    - t \E Z
    \,.
  \]
\end{lemma}
\begin{proof}
  Since $Z$ is non-negative, we have that $\E Z = \int_0^\infty \PR{Z \geq x}\intd
  x$, and by integration by parts (or non-negativity of the exponential) also that
  \[
    \E{\exp(tZ)} = 1 + \int_0^\infty t\exp\of{tx}\PR{Z\geq x}\intd x
  \]
  for all $t\in \R$.
\end{proof}

\begin{proposition}
  \label{prop:nonneg-cgf-expo-dom}
  Let $Z$ be a non-negative random variable stochastically dominated
  by the exponential distribution, i.e.~such that $\PR{Z\geq x}
  \leq \exp(-x)$ for all $x\geq 0$. Then for all $t\in(-\infty,1)$,
  it holds that
  \[
    \log\E{\exp\of{t\of{\centered Z}}} \leq B(t)
  \]
  where
  \[
    B(t) = \begin{cases}
      \frac{t^2}{1-t} &t\leq 0\\
      \max\set{
        \log\of{1+\frac{t^2}{1-t}-\frac{t^2}5},
        \log\of{1+\frac t5 + \frac{t^2}{1-t}}-\frac t5
      }
    &t\geq 0
  \end{cases}
  \]
  satisfies the upper bounds
  \begin{align*}
    B(t) &\leq
           \begin{cases}
             \frac{t^2}{1-t}&t\leq 0\\
             \log\of{1+\frac{t^2}{1-t}}&t\geq 0
           \end{cases} 
          \leq \frac{t^2}{1-t}
    &
      B(t)
    &\leq
      2\log\of{\frac{\exp(-t)}{1-t}}
  \end{align*}
  for all $t<1$.
\end{proposition}
\begin{remark}\label{rmrk:expo-const}
  By optimizing over the set of random variables stochastically dominated by the exponential,
  one can (with more work) establish an upper bound of the form $C
  \log\of{\frac{\exp(-t)}{1-t}}$ for an explicit constant $C < 2$, but since the
  result does not hold under the stated assumptions for $C = 1$, we do not attempt
  to optimize this constant beyond the minimal work we do here to give $C = 2$.
\end{remark}
\begin{proof}
  Note that $t\of{\exp(tx) - 1}\geq 0$ for all $x\geq 0$ and $t\in\R$, so that
  $t\of{\exp\of{tx}-1}\PR{Z\geq x}\leq t\of{\exp\of{tx}-1}\exp(-x)$, and thus by
  \cref{lem:nonneg-cgf-rep} we have for $t<1$ that
  \begin{align}
  \log\E{\exp\of{t\of{\centered Z}}}
  &\leq \log\of{1 + t\E Z + \int_0^\infty t\of{\exp\of{tx}-1}\exp(-x) \intd x}
  - t \E Z\nonumber\\
  &=\log\of{1 + t\E Z + \frac{t^2}{1-t}}
  - t \E Z\label{eqn:ez-upper-bound}
  \end{align}
  The upper bound for $t\leq 0$ follows from the fact that $\log(1+x)\leq x$ for
  all $x\in\R$. It remains to show the upper bound for $t\geq 0$, which we do
  in two cases based on $\E Z$.

  If $\E Z\geq 1/5$, then since \cref{eqn:ez-upper-bound} is decreasing in $\E
  Z$ (e.g.~by elementary calculus), we have that
  \[
    \log\E{\exp\of{t\of{\centered Z}}}
    \leq
    \log\of{1 + \frac t5 + \frac{t^2}{1-t}}
    - \frac t5
  \]
  as desired. On the other hand, if $\E Z\leq 1/5$, then by Markov's
  inequality we have for all $x\geq 0$ that $\PR{Z\geq x}\leq (\E Z)/x
  \leq 1/(5x)$, which is smaller than $\exp(-x)$ on an interval containing
  $[3/10,5/2]$. In particular, we can bound
  \begin{multline*}
    \int_0^\infty t\of{\exp\of{tx}-1}\PR{Z\geq x} \intd x
    \\
    \leq
    \int_0^\infty t\of{\exp\of{tx}-1}\exp(-x) \intd x
    -
    \int_{3/10}^{5/2} t\of{\exp\of{tx}-1}\of{\exp(-x) - 1/(5x)} \intd x 
    \,,
  \end{multline*} 
  where since $\exp(tx) - 1\geq tx$ we have
  \[
    \int_{3/10}^{5/2} t\of{\exp\of{tx}-1}\of{\exp(-x) - 1/(5x)} \intd x 
    \geq
    \int_{3/10}^{5/2} t^2 x\of{\exp(-x) - 1/(5x)} \intd x 
    \geq \frac{t^2}5
    \,.
  \]
  In particular, we get that
  \[
    \log\E{\exp\of{t\of{\centered Z}}}
    \leq
    \log\of{1 +t \E Z+ \frac{t^2}{1-t} - \frac{t^2}5} -t \E Z
    \leq
    \log\of{1 +\frac{t^2}{1-t} - \frac{t^2}5}
    \,
  \]
  where the second inequality is because the function is decreasing in
  $\E Z$.

  Finally, we prove the upper bounds on $B$. The first bound follows
  from the fact that $\log$ is an increasing function, $\log(1+x)\leq x$ for all
  $x$, and that $\log(C + x) - x$ is a decreasing function of $x\geq 0$ for
  $C\geq 1$. For the second bound, elementary calculus shows that
  \[
    2\log\of{\frac{\exp(-t)}{1-t}} - B(t)
  \]
  is non-increasing on the non-positive reals and non-decreasing on the
  non-negative reals, so that since it is $0$ at $0$ the bound follows.
\end{proof}

\subsection{Putting it together}
We can now prove the main results as stated in the introduction.
\begin{theorem}[\cref{thm:mgf} restated]
  For every $n$, $k$, and $P$, we have for all $t<n/2$ that
  \[
    \log\E{\exp\of{t\of{\centered{\empkl nkP}}}}
    \leq
    \min\set{
      \frac{4kt^2/n^2}{1-2t/n}
      ,\,
      2k \log\of{\frac{\exp\of{-2t/n}}{1-2t/n}}
    }
  \]
\end{theorem}
\begin{proof}
  \Cref{prop:multired,cor:redqty-domexpo} show that the centered moment generating
  function of $\empkl nkP$ at $t\in \R$ is dominated by
  \[
    \sqrt{\prod_{i=1}^{k}\exp\of{\frac{2t}{n}\cdot\of{\centered{Z_{i}}}}}
    \cdot
    \sqrt{\prod_{i=1}^{k}\exp\of{\frac{2t}{n}\cdot\of{\centered{Z'_{i}}}}}
  \]
  where $Z_{i}$ and $Z'_{i}$ are non-negative random variables stochastically
  dominated by the exponential distribution, so the result follows from
  \cref{prop:nonneg-cgf-expo-dom}.
\end{proof}

\begin{corollary}[\cref{cor:tail} restated]
  For all $\eps\geq 0$ we have that
  \begin{multline*}
    \PR{\empkl nkP \geq \E{\empkl nkP} + \eps}
    \leq \of{1 + \frac{n\eps}{4k}}^{2k}\cdot \exp\of{-\frac{n\eps}2}
    \\
    \leq \exp\of{-\frac{3n^2\eps^2}{48k+8n\eps}}
    \leq \exp\of{-\min\set{\frac{n^2\eps^2}{24k},\frac{n\eps}8}}
  \end{multline*}
  and for all $0\leq\eps \leq 2k/n$ we have that
  \[
    \PR{\empkl nkP \leq \E{\empkl nkP}-\eps}
    \leq \exp\of{-k\of{1-\sqrt{1-\frac{n\eps}{2k}}}^2}
    \leq \exp\of{-\frac{n^2\eps^2}{16k}}
    \,.
  \]
  In particular, \cref{conj:tb} holds with $c_1=2$ and $c_{2}=1/48$.
\end{corollary}
\begin{proof}
  The first inequality in each chain is immediate from \cref{thm:mgf} by
  computing the optimal Chernoff bound from (i.e.~the convex conjugate of)
  $2k \log\of{\frac{\exp\of{-2t/n}}{1-2t/n}}$ for the upper tail and
  $\frac{4kt^2/n^2}{1-2t/n}$ for the lower tail. The relaxed bounds
  follow from the elementary inequalities $\log\of{1+x}\leq\frac x2
  \cdot \frac{x+6}{2x+3}$ (e.g.~\cite{topsoe_bounds_2007}) for $x\geq 0$
  and $1 - \sqrt{1-x} \geq x/2$ for $x\leq 1$. The implication for
  \cref{conj:tb} is because $k\geq 2$ implies $k \leq 2(k-1)$.
\end{proof}

\begin{corollary}[\cref{cor:moments} restated]
  We have that $\Var\of{\empkl nkP}\leq 8k/n^2$, and more generally, for
  all integers $m\geq 1$ we have
  \[
    \E{\of{\centered{\empkl nkP}}^{2m}}\leq
    \frac{2^{6m}\of{k^mm! + (2m)!}}{n^{2m}}
    \,,
  \]
  so that in particular for all real $q\geq 1$ we have that
  \[
    \sqrt[q]{\E{\abs{\centered{\empkl nkP}}^{q}}}\leq
    \frac{24}{n}\of{\sqrt{kq} + q}\,.
  \]
\end{corollary}
\begin{proof}
  The variance bound follows from the fact that
  $\Var\of Z = \lim_{t\to 0}\frac{ \log\E{\exp\of{t\of{\centered Z}}} }{t^2/2}$
  for a random variable $Z$ with moment generating function finite around $0$, and
  the general claim for integer $m$ follows from standard results on sub-gamma
  random variables, e.g.~\cite[Theorem 2.3]{bou_lug_mas_concentration_2013}
  applied to the bound from \cref{thm:mgf}.

  The in particular claim follows because
  $ \sqrt[q]{\E{\abs{\centered{\empkl nkP}}^{q}}}$ is a non-decreasing function of
  $q$ by Jensen's inequality, so we have that if $2m\leq q + 2$ is the
  smallest even integer at least $q$, then
  \begin{multline*}
    \sqrt[q]{\E{\abs{\centered{\empkl nkP}}^{q}}}
    \leq
    \frac{8}{n}\sqrt[2m]{k^{m}m! + (2m)!}\\
    \leq\frac{8}{n}\of{\sqrt[2m]{k^mm^{m}}+\sqrt[2m]{(2m)^{2m}}}
    \leq \frac{24}{n}\of{\sqrt{kq} + q}
  \end{multline*}
  where the last line is because $q\geq 1$ and $2m\leq q + 2$ implies $2m\leq 3q$.
\end{proof}

\printbibliography
\end{document}